\documentclass[11pt]{amsart}
\usepackage{amsmath}
\usepackage{amsthm}
\usepackage{amsfonts}
\usepackage{amssymb}
\usepackage{verbatim}
\usepackage{graphicx}
\usepackage[margin=1.0in]{geometry}
\usepackage{url}
\usepackage{enumerate}
\usepackage[pdftex,bookmarks=true]{hyperref}

\newcommand\R{{\mathbb{R}}}

\renewcommand\P{{\mathbf{P}}}
\newcommand\E{{\mathbf{E}}}

\newcommand\CB{{\mathcal B}}

\newcommand\CE{{\mathcal E}}

\newcommand\ep{{\epsilon}}

\newcommand{\ab}[1]{\left| #1 \right|}

\theoremstyle{plain}
  \newtheorem{theorem}{Theorem}

  \newtheorem{lemma}[theorem]{Lemma}
  
  \newtheorem{claim}[theorem]{Claim}

\theoremstyle{definition}

  \newtheorem{remark}[theorem]{Remark}

\include{psfig}

\begin{document}

\title {On the number of real roots of  random polynomials}

\author{Hoi Nguyen} 
\address{Department of Mathematics, The Ohio State University, 231 West 18th Avenue, Columbus, OH 43210}
\email{nguyen.1261@math.osu.edu}
\thanks{H. Nguyen is supported by research grant DMS-1358648}

\author{Oanh Nguyen}  
\address{Department of Mathematics, Yale University, New Haven, CT 06520, USA}
\email{oanh.nguyen@yale.edu}

\author{Van Vu}
\thanks{O. Nguyen and V. Vu is supported by research grants DMS-0901216 and AFOSAR-FA-9550-09-1-0167. }
\address{Department of Mathematics, Yale University, New Haven , CT 06520, USA}
\email{van.vu@yale.edu}

\begin{abstract} 
Roots of random  polynomials have been studied exclusively in both analysis and probability for a long time. 
A famous result by Ibragimov and Maslova, generalizing earlier  fundamental works of Kac and Erd\H os-Offord, showed that the expectation of the number of real roots  is  $\frac{2}{\pi} \log n + o(\log n)$.   In this paper,  we determine the true nature of the error term by showing that 
the expectation equals  $\frac{2}{\pi}\log n + O(1)$. Prior to this paper, such estimate has been known only in the gaussian case, thanks to works of 
 Edelman and Kostlan. 
\end{abstract}

\maketitle

\section{Introduction}
Consider a random polynomial $P_{n, \xi}  (z) = \sum_{i=0}^n  \xi_i x^i$ where $\xi_i$ are iid copies of a real random variable $\xi$ with mean zero.
Let $N_{n, \xi} $ denote the number of real roots of $P_{n, \xi}$. In what follows the asymptotic notations are used under the assumption that 
$n \rightarrow \infty$; notation such as $O_k (1)$ means that the hidden constant in big "O"  may depend on a given parameter $k$.

Waring was the first to investigate  roots of random polynomials as far back as 1782 (see, for instance,  Todhunter's book on early  history of probability \cite[page 618]{To}, which also mentioned 
a similar contribution of Sylvester).  As customary in those old days, the source of randomness was not specified in these works.   More rigorous and 
systematic studies of $N_{n, \xi}$ started in the 1930s. 
In 1932,  Bloch and P\'olya \cite{BP} considered the special case when $\xi$ is  uniformly distributed in $\{-1,0,1\}$ and established the upper bound
$$ \E N_{n, \xi}  = O( n^{1/2}).$$  Their method can be extended to other discrete distributions such as Bernoulli ($\xi =\pm 1$ with probability $1/2$); see \cite{erd}. 
  This bound is not sharp,  and it was a considerable surprise when Littlewood and Offord showed that random  polynomials actually have a remarkably small number of real zeroes.  
 In a  series of  fundamental  papers \cite{lo-2,lo-3, lo-4}, published between  1939 and 1945,  they  proved  a strong bound 
\begin{equation}  \label{Kac1}  \frac{\log n}{\log \log \log n} \ll N_{n, \xi}  \ll \log^2 n \end{equation} 
with probability $1-o(1)$,  for many basic variables  $\xi$ (such as Bernoulli, Gaussian, and uniform on $[-1,1]$). 

During this time, in 1943, another fundamental  result was achieved by  Kac  \cite{kac-0}, who  found an asymptotic  estimate for $\E N_{n, \xi}$ in the case that $\xi$ is standard real Gaussian $N(0,1)$, showing 
\begin{equation}  \label{Kac2}  \E N_{n, N(0,1)} = \left(\frac{2}{\pi} +o(1)\right) \log n.  \end{equation}

It took much effort   to extend \eqref{Kac2} to other distributions.  Kac's method does provide a  formula for $\E N_{n, \xi}$ for any $\xi$. However, 
this formula is  hard   to estimate when $\xi$ is not Gaussian.  In a subsequent paper \cite{kac-1}, Kac managed to extend  \eqref{Kac2} 
to the case when $\xi$ is uniform  on $[-1,1]$ and  Stevens \cite{Stev}  extended it further  to cover  a large class  of  $\xi$ having continuous and smooth  distributions
with certain  regularity properties (see \cite[page 457]{Stev} for details). These papers
relied  on Kac's formula and the analytic properties of the distribution of $\xi$ are essential.  
({\it A historical remark} : In \cite{kac-0}, Kac was very optimistic and thought that his argument would work for all random variables. However, 
he soon realized that it was not the case, and his proof for the  uniform case was already substantially more complicated than 
that of the gaussian case; see \cite{kac-1}.)

For random variables with no analytic properties, it is a completely different ball game. 
 Since Kac's paper, it  took sometime until Erd\H{o}s and Offord  in 1956 \cite{EO} found a 
 new approach  to handle 
discrete distributions.  Considering the case when $\xi$ is Bernoulli, they proved that with probability $1- o(\frac{1} {\sqrt {\log \log n} } )$

\begin{equation} 
N_{n, \xi}  = \frac{2}{\pi} \log n + o(\log^{2/3} n \log \log n). 
\end{equation} 

Erd\H os often listed this result among his favorites achievements (see, for instance \cite{erdossurvey}). 
In   late 1960s and early 1970s,  Ibragimov and Maslova \cite{IM1, IM2}  successfully refined  Erd\H{o}s-Offord method to handle 
any variable $\xi$ with mean 0. They  proved  that for any $\xi$ with mean zero which belong to the domain of attraction of the normal law,

\begin{equation}  \label{IM-1} 
\E  N_{n, \xi}  = \frac{2}{\pi} \log n + o(\log n) .
\end{equation} 

The error term $o(\log n)$ is implicit in their papers. However, 
by following the proof (see the last bound in \cite[page 247]{IM1}) it seems that one  can replace it by a more precise term $O( \log^{1/2} n \log \log n )$. 
For related results, see also \cite{IM3, IM4}. 
Few years later, Maslova  \cite{Mas1, Mas2} showed that if $\xi$ has mean zero and variance one and $\P (\xi =0) =0$, then the variance of $N_{n, \xi}$ is 
$(\frac{4}{\pi} (1- \frac{2}{\pi} )+o(1) ) \log n$.

Fast forwarding twenty more years, one records  another important development,  made by Edelman and Kostlan \cite{EK}  in 1995. They 
 introduced a new way to handle the Gaussian case and  estimate $\E N_{n, N(0,1)} $. 
Using delicate analytical tools, they  proved  the following  stunningly precise formula

\begin{equation}    \label{Kac6} 
  \E N_{n, N(0,1) }  =  \frac{2}{\pi} \log n  + C_{N(0,1)} +  \frac{2}{\pi n} +O(\frac{1}{n^2}) 
\end{equation}

\noindent where $C_{N(0,1)} \approx .625738072..$ is an explicit  constant ( it is the value of an explicit, but complicated, integral).

The approach used in \cite{EK} relies  critically   on the fact that a random Gaussian vector distributes uniformly on the unit sphere and  cannot be used 
 for other distributions.  The true nature of the error term in $\E (N_{n, \xi}) $  has not been known  in general and all  of the existing approaches  lead to 
 error term polynomial in $\log n$ . In particular,  it seems  already  very difficult to 
improve upon  the order of magnitude of the error term in Ibragimov and Maslova's analysis.

In this paper, we provide a new method to estimate $\E N_{n, \xi}$.  This method enables us to derive the following sharp estimate

\begin{theorem} \label{theorem:main}  For any random variable $\xi$ with mean 0 and variance 1 and bounded $(2 +\epsilon)$-moment
$$ \E N_{n, \xi}  =\frac{2}{\pi} \log n +O_{\epsilon, \xi} (1). $$
\end{theorem}

Without loss of generality, we will assume $\ep$ to be sufficiently small.  To emphasize the dependence of the hidden constant in  big $O$  on the atom variable  $\xi$,  let  us notice that if $\xi$ is Bernoulli, then the random polynomial  $P_{n, \xi} $ does not have any real root in the interval $(-1/2, 1/2)$ with probability 1. On the other hand, 
one can show that if  $\xi$ is gaussian then the expectation of number of  real roots in $(-1/2, 1/2)$ is $\frac{1}{\pi}\log 3 + o(n^{-17})$.  Thus, it is reasonable to expect  that 
the expectation in the Gaussian case exceeds that in the Bernoulli case by  a positive constant. Our numerical experiment tends to agree with this.

\begin{figure}[ht!]
\centering
\includegraphics[width=150mm]{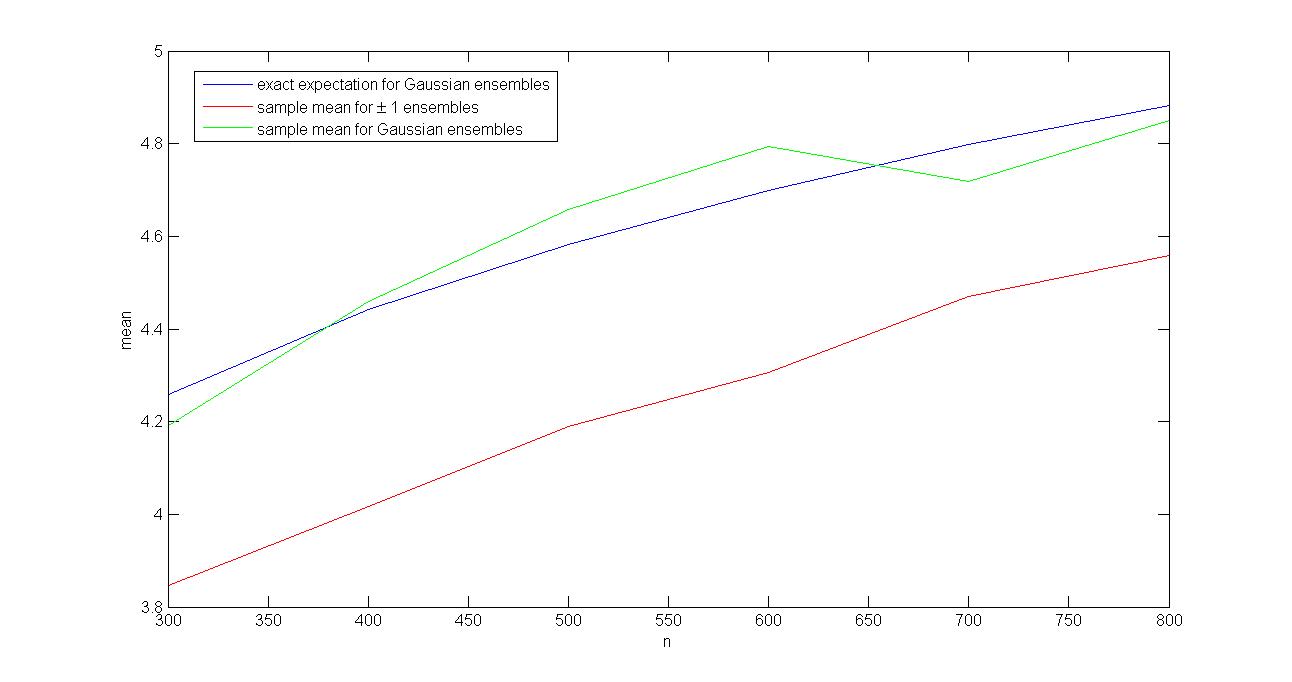}
\caption{Sample means of the number of real roots for Gaussian and Bernoulli ensembles; the $x$-axis represents the degree. It seems that the expectation in the Gaussian cases exceeds that in the Bernoulli case by roughly $.4$.}
\end{figure}

Theorem \ref{theorem:main} is a corollary of a stronger theorem, which provides an  even more satisfying  estimate on the main part of the spectrum. For any region $D\subset \R$, let $N_{n, \xi} D$ denote the number of 
real roots in  $D$. It is well known (see for instance \cite{BSbook,Far})  that one can reduce the problem of estimating $N_{n, \xi}$ to $N_{n, \xi} [0,1]$; as a matter of fact 

$$\E N_{n, \xi} =  4 \E N_{n, \xi} [0,1] . $$

Inside the interval $[0,1]$, most of the real roots are  clustered near  1.  For any constant $C$, the number of roots between $0$ and $1 -C^{-1} $ is only $O_{C} (1) $. More precisely,

%
%
%
%

\begin{lemma}\label{fact:edge}  
 For any positive constant $C$, there exists a constant $M(C)$ such that 
\begin{equation} \label{Kac11-0} \E N_{n, \xi} [0, 1- C^{-1} ) \le M(C). \end{equation}

%

Furthermore, there exists a constant $ C_0$ such that for any $C$ greater than $C_0$, 

\begin{equation} \label{Kac11-1}   \E N_{n, \xi} [0, 1- C^{-1} )  \le \frac{1}{2\pi}   \log C + M(C_0).
 \end{equation}
\end{lemma}  
Notice that $C$ is allowed to depend on $n$ in \eqref{Kac11-1}. Thus (by taking $C$ to be, say, $100 n$)  \eqref{Kac11-1} almost gives the upper bound in Theorem \ref{theorem:main}. As well known in this area, the lower bound is 
often the heart of the matter.


Let us  focus on   the bulk of the spectrum,  the interval $(1- C^{-1} , 1]$.  For this part,  we obtain the following precise estimates, regardless 
the nature of the atom variable $\xi$. 

\begin{theorem}\label{theorem:main1} There exists a constant $C_0$ such that for  any variable $\xi$ with mean 0, variance 1,  and bounded 
$(2+\epsilon)$-moment
\begin{equation}   \label{Kac10} 
\left |\E N_{n, \xi}  (1- C_0 ^{-1} , 1]  - \int_{1- C_0^{-1} }^1  \frac{1}{\pi} \sqrt { \frac{1}{(t^2-1)^2} - \frac{(n+1)^2 t^{2n} }{ (t^{2n+2} -1)^2 } } dt \right |
 \le   C_0 ^{-\ep}.  
 \end{equation} 
 
 Furthermore, for any number $C \ge C_0$, there exists  $C'\in [C^\ep, C]$  such that 
\begin{equation}   \label{Kac11} 
\left |\E N_{n, \xi}  (1- {C'} ^{-1} , 1]  -   \int_{1- {C'}^{-1} }^1  \frac{1}{\pi} \sqrt { \frac{1}{(t^2-1)^2} - \frac{(n+1)^2 t^{2n} }{ (t^{2n+2} -1)^2 } } dt   \right | \le   {C'}^{-1 }.  
\end{equation} 

 \end{theorem}

%





The  integral on the LHS of \eqref{Kac10} is the explicit formula for 
 $\E N_{n, N(0,1) }  (1- C_0^{-1} , 1] $ (see \cite{EK}). Thus, one can rephrase \eqref{Kac10} as 

\begin{equation}   \label{Kac12} 
\left | \E N_{n, \xi}  (1- C_0^{-1} , 1]  -   \E N_{n, N(0,1)}  (1- C_0^{-1} , 1]  \right | \le   C_0 ^{-\ep}.  
\end{equation} 

This, combining with the argument following Theorem \ref{theorem:main},  reveals an interesting fact that the 
 impact of the distribution of  $\xi$  is  felt  only at the left  "edge" of the spectrum.

Theorem \ref{theorem:main} follows immediately  from Lemma \ref{fact:edge} and Theorem \ref{theorem:main1}.  Our proofs are  quantitative and in principle one can 
derive an explicit value for $C_0$. However, this involves a  tedious book keeping and in general we do not try to optimize the constants in this paper.  
Our proof also shows that \eqref{Kac12}  still  holds if we replace the interval $(1- C_0^{-1}, 1]$ by any subinterval. Furthermore, our approach, which  makes use of 
a recent  universality result from \cite{TVpolynom}  and the non-existence of near double roots,  is entirely different from  previous approaches.

\begin{remark} \label{remark:general}  We would like to point out an important fact that our results hold, without any significant modification in the proof,  
for more general settings where  the variables 
$\xi_i$ in the definition of $P_n$ are not iid. It suffices to assume that they all have mean 0, variance 1,  and uniformly bounded $(2+\epsilon)$-moments.  
\end{remark} 



\section{Number of real roots in an interval very close to 1} 

Our starting point is the following theorem, which  is a corollary of   \cite[Theorem 25]{TVpolynom}.

%

\begin{theorem}  \label{theorem:TV} There is a positive constant $\alpha$ such that the following holds. 
Let $\epsilon >0$ be an arbitrary small constant and  $\xi$ be any random variable with mean zero, variance one and bounded $(2+\epsilon)$-moment. Then
there is a constant $C_1 := C_1 (\epsilon) $ such that for any  $n \ge  C_1$ and any interval $I := (1-r, a) \subset (1- n^{-\epsilon} , 1]$

\begin{equation}  \label{small2} 
|\E _{n, \xi} I   - \E _{n, N(0,1) } I  | \le n ^{-\alpha}.
\end{equation} 
\end{theorem}

This is close, in spirit, to \eqref{Kac12}. The main technical obstacle here is that the result holds only in a region polynomially close to 1. 
The key  new ingredient  we have in this paper is the observation that 
a random polynomial, with high probability, does  not have double   or  near double roots. 
 We discuss this  observation, which is  of independent interest, in the next section.  At the end, we can prove \eqref{Kac12} by  combining (a sufficiently quantitative version of) this observation with Theorem \ref{theorem:TV}. The proof of Lemma \ref{fact:edge}, which is independent from the main argument, is provided at the end of the paper.

\section{Non-existence of near double roots}

A double root $\lambda$ satisfies $P_n (\lambda) = P_n' (\lambda) =0$.  We introduce a more  general notion of {\it near double roots}: $\lambda$ is  a  near double root if $P_n (\lambda) =0$ and 
$| P'_n (\lambda)|$ is small.  Existence of double roots and near double roots are of interest in analysis and numerical analysis  (see for instance 
the studies of Newton's method for finding real roots \cite{BCSS}). 

Our new tool is  the following lemma, which asserts that there are no near double roots in the bulk of the spectrum with high probability.


\begin{lemma}\label{lemma:dbroot}  For any constant $C > 0$, there exist $B=B(C),B_0=B_0(C)$, and $B_1=B_1(C)$ such that
  
$$\P\Big( \exists x\in (1- B_0^{-1}, 1 - \frac{ B_1 \log n}{n} ]: P_n(x)=0, |P_n'(x)| \le n^{-B}\Big) = o(n^{-C}) .$$
\end{lemma}

\subsection{Preliminaries}  To start, we deduce a property of  polynomials having a near double  root. 
Let $\delta$ be a small parameter to be chosen and $Q \subset (1/2,1]$ be an interval of length $2\delta$  centered at $x_Q$. 
If there is  $x\in Q$ such that $P_n(x)=0$  then by the mean value theorem
$|P_n(x_Q)|\le \delta |P_n\rq{}(y)|$ for some $y $ between $x$ and $x_Q$.   (We can write $\delta/2$ instead of $\delta$ on the RHS, but this does not make any difference.)

Assume that $|P'_n (x)| \le n^{-B}$, then by applying the mean value theorem again, we have $|P'_n (y) |\le \delta |P^{''}_n (z) |+ n^{-B}$ for some $z$ between $x$ and $y$.  Furthermore, with a loss of a probability bound $O(n^{-C-1})$, one can assume that $|\xi_i|\le n^{C/2+1}$ for all $i$, and so $|P''(z)|\le n^{4+C/2}$. Thus,
  
 \begin{equation}\label{eqn:bound}
|P_n(x_Q)| \le \delta^2 n^{4+C/2} + \delta n^{-B}.
\end{equation}

Set $\delta :=n^{-A}$ for some suitable constant $A$ to be chosen, and $B:= A-C/2-2$  so  that the term $\delta^2n^{4+C/2}$ dominates. We partition the interval $I:=(1-B_0^{-1}, 1 - \frac{C_1 \log n}{n} ]$ into subintervals $I_i=(1- B_0^{-1}+(i-1)n^{-A-3}, 1- B_0^{-1}+(i+1)n^{-A-3}]$ with center $x_i=1- B_0^{-1}  + i n^{-A-3}$ and length $\delta$ and show that with high probability \eqref{eqn:bound}  fails at every center. 

\subsection{Small ball estimate}

Set $\gamma:= 2\delta^2 n^{4+C/2}$, we are going to prove the following small ball estimate.

\begin{lemma}\label{lemma:ball}
For any $1-B_0^{-1}<x< 1-B_1\log n/n$, one has 

$$\P(|P(x)|\le \gamma) =O({\gamma}^{.99}).$$
\end{lemma}

In order to prove this theorem, we first need the following elementary claim whose proof is left as an exercise. 

\begin{claim}\label{claim:LO} There exist positive constants $c_0$ and $p_0$ (depending on $\ep$) such that for any $\xi$ of mean 0, variance 1, and bounded $(2+\ep)$-moment, there exists $c_0\le c \le c_0^{-1}$ such that 

$$\P(c < |\xi-\xi' | < 2c) \ge p_0.$$
\end{claim}

By switching from $\xi$ to $\xi/c$ if needed, without changing the result of Lemma \ref{lemma:ball}, one can assume that 

$$\P(1<|\xi-\xi'|<2)\ge p_0.$$

\begin{proof}[Proof of Lemma \ref{lemma:ball}] Let $\xi'_1,\ldots,\xi'_{n}$ be independent copies of $\xi_1,\ldots,\xi_{n}$, let $\epsilon_1,\ldots,\epsilon_{n} \in \{-1,1\}$ be independent Bernoulli variables (independent of both $\xi_i$ and $\xi'_i$), and let $\tilde \xi_i$ be the random variable that equals $\xi_i$ when $\epsilon_i = +1$ and $\xi'_i$ when $\epsilon_i = -1$.  Then $\tilde \xi_1,\ldots,\tilde \xi_{n}$ have the same joint distribution as $\xi_1,\ldots,\xi_{n}$, so it suffices to obtain the bound

$$ \P( |\sum_{i=0}^n \tilde \xi_i x^i| \leq \gamma ) =O( \gamma^{.99}).$$

Let $\delta_0>0$ be sufficiently small ($\delta_0=.000001$ would suffice) and $t_0$ be such that 

\begin{equation}\label{eqn:t0}
(1-p_0)^{t_0} <\delta_0.
\end{equation}

Let $N$ be chosen so that $x^{N+1}$ is approximately $\gamma$ (such as $2\gamma \le x^{N+1}\le 4\gamma$ would suffice). Notice that as $1-B_0^{-1}\le x \le 1-B_1\log n/n$ with sufficiently large $B_0,B_1$, we have

$$\Omega(\log n) \le N \le n.$$  

Without loss of generality we assume that $N+1$ is divisible by $t_0$. Divide the set $\{0,1,\dots, N\}$ into $m=N/t_0$ intervals $J_1,\dots,J_m$ with $J_1:=[0,\dots,t_0-1], J_2:=[t_0,\dots,2t_0-1],\dots, J_m:=[N-t_0,\dots, N]$. 

Let $J \subset \{1,\ldots, m\}$ be a (random) subset of indices $k$ for which the following holds for at least one index $i$ from $J_k$,

\begin{equation}\label{sss}
 1  < |\xi_{i} -\xi'_{i} | <2.
\end{equation}

By definition, we have 

$$\P(k\in J) \ge 1 - (1-p)^{t_0} \ge 1-\delta_0.$$ 

Let $\CE$ be the event that $|J|\geq m':=(1-2\delta_0)m$. From Chernoff's lower tail bound, one has

\begin{equation}\label{E}
\P(\CE^c) \le 2 \exp(-\frac{\delta_0^2}{2}m).
\end{equation}

As $x\ge 1-B_0^{-1}$ and $B_0$ is sufficiently large, we have

$$\exp(-\frac{\delta_0^2}{2}m)=\exp\big(-\frac{\delta_0^2}{2t_0}(N+1)\big) \le (1-B_0^{-1})^{N+1}\le x^{N+1} = O(\gamma),$$

where we recall that $x^{N+1}$ is approximately $\gamma$.

From now on we condition on $\CE$, thus assuming 
\begin{equation}\label{eqn:m'}
m'\ge (1-2\delta_0)m.
\end{equation}

By considering a subset of $J$ if needed, one can assume that $|J|=m'$. From each interval $J_k$ where $k\in J$, we choose one single index $i\in J_k$ such that \eqref{sss} holds. In what follows we will fix the random variables $\xi_i, \xi_i'$ for all $i$; and the signs $\epsilon_i$ if $i$ was not chosen.

In summary, one obtain subsequences $1\le i_1<\dots < i_{m'} \le m$ and $0\le n_1 < n_2<\dots < n_{m'} \le N$ with the following properties:

\begin{itemize}
\item $1<|\xi_{n_j}-\xi_{n_j}'|<2$;
\vskip .1in
\item $n_j \in J_{i_j}$.
\vskip .1in
\item The (only) source of randomness comes from the sign $\ep_{n_1},\dots, \ep_{n_{m'}}$.
\end{itemize}

Set 

$$y:= x^{t_0}.$$

By definition, as $n_j\in J_{i_j}=[(i_j-1)t_0,\dots, i_j t_0-1]$, one has the following double bound

\begin{equation}\label{eqn:sandwich}
y^{i_j}<x^{n_j} \le y^{i_j-1}.
\end{equation}

As $1- B_0^{-1}  < x < 1-B_1\log n/n$, there is a unique positive integer  $l\le m$ such that 

\begin{equation} \label{eqn:boundxi} 
1/4 < y^{l}< 1/2 \le y^{l-1}. 
\end{equation} 

Furthermore, since $B_0$ is sufficiently large, one has the following elementary bound

\begin{equation}\label{eqn:l}
l\ge 1000.
\end{equation}

Let $k$ be the largest integer such that $(l+2)k\le m$. Thus

\begin{equation} \label{eqn:bounds:y}  
y^{(l+2)k} \ge y^m = x^{t_0 m} = x^{N+1} \ge 2 \gamma  \mbox{ and } y^{(l+2)(k+1)} \le y^m = x^{N+1}  \le 4\gamma. 
\end{equation} 

Again, because $B_0$ is sufficiently large, $y^{l+2}=y^3 y^{l-1}\ge y^3/2 > 1/4$. Thus as $y^{(l+2)(k+2)}\le 4\gamma$, $k$ must have order at least $\Omega(\log n)$. This yields the following elementary bound (assuming $n$ sufficiently large),

\begin{equation}\label{eqn:k}
k\ge 1000.
\end{equation}

Let $S$ be the subset of multiples of $l+2$ in \{0,\dots, m\}, $S:=\{0,l+2, \dots, \lfloor m/(l+2) \rfloor (l+2) \}$. Consider the decomposition $\{0,\dots,m\}$ into $S\cup (S+1)\cup \dots \cup S+(l+1)$. By \eqref{eqn:m'} and by the pigeon hole principle, there exists $i_0\le l+1$ such that 

\begin{equation}\label{eqn:Si_0}
|S+i_0 \cap \{i_1,\dots,i_{m'}\}|\ge (1-2\delta_0) m/(l+2).
\end{equation}

We now work with the partial sum of $x^{n_j}$ with $i_j \in S+i_0$. To do this, we first introduce an elementary property of Bernoulli sums. 

Given a quantity $t>0$, we  say that a  set $X$ of real numbers is {\it $t$-separated } if the distance between any two elements of $X$ is at least $t$.

\begin{claim}\label{claim:separation:1}
The set $ \Big\{\sum_{1\le i \le k} \ep_i y^{i(l+2)}, \ep_i\in \{-1,1\} \Big \} $ is $2y^{k(l+2)}$-separated. 
\end{claim}

\begin{proof}[Proof of Claim \ref{claim:separation:1}]
Assume that there are two terms within distance smaller than $2y^{k(l+2)}$. Consider their difference, which has the form  $2(\ep_{m_1} y^{m_1(l+2)} + \dots+ \ep_{m_j} x_i^{m_j(l+2)})$ for some $m_1<\dots<m_j \le k$. As $y^{l+2} < y^l <1/2$, this difference in absolute value is at least 

 \begin{equation} \label{eqn:difference:comparison}
 2(y^{m_1 (l+2)} -y^{m_2 (l+2)} - \dots- y^{ m_j (l+2)} ) \ge 2y^{k(l+2)},
 \end{equation}
 
 a contradiction.
\end{proof}

By following the same argument,  one obtains the following.

\begin{claim}\label{claim:separation:2}
The set $ \Big\{\sum_{i_j \in S+i_0} \ep_j x^{n_j}, \ep_i\in \{\xi_{n_j},\xi_{n_j}'\} \Big \} $ is $2 y^{i_0} y^{k(l+2)}$-separated. 
\end{claim}

\begin{proof}[Proof of Claim \ref{claim:separation:2}] Recall that by our conditioning, 

$$1<|\xi_{n_j}- \xi_{n_j}'| <2.$$ 

Furthermore, if $i_j <i_{j'} \in S+i_0$ then $i_{j'} \ge i_j +l+2$. So, by \eqref{eqn:sandwich}

$$\frac{x^{n_{i_{j'}}}}{x^{n_{i_{j}}}} \le \frac{y^{i_{j'}-1}}{y^{i_{j}}} < y^{l} < 1/2.$$
\end{proof}

We now finish the proof of Lemma \ref{lemma:ball}. By Claim \ref{claim:separation:2}, and by the bound \eqref{eqn:Si_0}

\begin{equation}\label{eqn:bound:condition}
\sup_{R \in \R} \P\Big(|\sum_{i_j \in S+i_0} \tilde \xi_{n_j} x^{n_j}+R| \le 2 y^{i_0}y^{k (l+2)}\Big)\le 2^{-(1-2\delta_0)k}.
\end{equation}

Using $2y^{i_0} y^{k(l+2)}\ge 2 y^{(k+1)(l+2)}\ge \gamma$, we obtain

\begin{equation}\label{eqn:bound:condition}
\sup_{R \in \R} \P\Big(|\sum_{i_j \in S+i_0} \tilde \xi_{n_j} x^{n_j}+R| \le \gamma\Big)\le 2^{-(1-2\delta_0)k}.
\end{equation}

Consider the probability bound on the RHS.  Notice from  \eqref{eqn:l}  and \eqref{eqn:k} that both $k$ and $l$ are at least $1000$. So, $(l-1)k \ge .999 (k+1)(l+2)$. Thus, with $\delta_0=.000001$ and recall that $1/2\le y^{l-1}$

\begin{equation} \label{eqn:bond1} 
2^{-(1-2\delta_0)k}\le y^{(1-2\delta_0)(l-1)k} \le y^{(1-2\delta_0).999(k+1)(l+2)} \le    \gamma^{.99},
\end{equation} 

where we used \eqref{eqn:bounds:y} in the last estimate, assuming $n$ sufficiently large.
\end{proof}

We now complete the proof of our main result.

\begin{proof}[Proof of Lemma \ref{lemma:dbroot}]
Since there are  less than  $\delta^{-1}$ intervals, it follows from Lemma \ref{lemma:ball} and by the union bound,
 
$$\P(\exists i, |P(x_i)|\le  \gamma )  \le \delta^{-1} \gamma^{.99} = \delta^{.98} n^{.99(4+C/2)} = o( n^{-.98 A + 4+C/2}) . $$  

By setting $A := 3C+6$ and recall our choice $B=  A-C/2-2$, we have 

$$\P\Big( \exists x\in (1- B_0^{-1}, 1 - \frac{ B_1 \log n}{n} ]: P_n(x)=0, |P_n'(x)| \le n^{-5C/2 -4}\Big) = o(n^{-C}),$$

proving the desired statement. 
  
\end{proof}

 \begin{remark}
 It follows from our proof that instead of having bounded $(2+\ep)$-moment, it suffices to assume that there exist positive constants $c_1,c_2$ and $p$ such that 
 
 $$\P(c_1<|\xi-\xi'|<c_2)\ge p.$$  
 \end{remark}

\begin{remark}\label{remark:stronger} 
We can also extend our argument, with few modifications, to show the non-existence of near double roots in $(1-B_0^{-1}, 1]$ for general $\xi$, and in  the whole spectrum for Bernoulli polynomials; details will follow in a subsequent paper. 
\end{remark}

Using a similar  argument (with the same definition of $\delta$ and $I$)  we can prove the following.

\begin{lemma}  \label{lemma:distance}  For any constant $C >1$,  the following holds with probability  $1 -o(n^{-C})$. 

\begin{itemize} 

\item There is no pair of roots in $I=(1-B_0^{-1},1-B_1\log n/n)$ with distance at most $\delta =2 n^{-3C-6}$. 
\vskip .1in
\item For any given $a \in I$, there is no root with distance at most $\delta' := n \delta^2$ from $a$. 
\end{itemize} 
\end{lemma} 

\begin{proof} [Proof of Lemma \ref{lemma:distance}]

For the first statement, we can fix a $\delta$-net $S =\{x_1, \dots, x_M \}$ on $I$  such that for any $x \in I$, there is some $x_i \in S$ with distance at most $\delta$ to $x$ and $M \le \delta^{-1} +1$.

If $P_n (x) =P_n (x') =0$, then there is a point $y$ between $x$ and $x'$ such that $P_n'(y)=0$. Thus, for any $z$ with distance at most $2\delta$ from $y$,  

$$|P'(z)| \le 2 n^{4+C/2} \delta.$$ 

There is a point $x_i$ in the net such that $|x_i-x| \le \delta$. For this $x_i$,  $|P_n(x_i)| = |x_i-x| | P'_n (z)| $ for some $z$ between $x$ and $x_i$. Because $x$ has distance at most $\delta$ from $x'$, $x$ also has distance at most  $\delta$ from $y$, and so $z$ has distance at most $2\delta$ from $y$. It follows that

$$|P_n (x_i) |   \le 2 n^{4+C/2} \delta ^2.$$

From the previous proof, the probability that the above bound holds for some $i$ is $o(n^{-C})$.

For the second statement, assume that $P_n (x) =0$ and $|a-x| \le \delta' $, then $|P_n(a) | = |a-x| |P_n' (y)| $ for some $y$ between $a$ and $x$. On the other hand,  with a loss of $n^{-C-1}$ in probability, one can assume that $|P_n'(y) | \le n^{3+C/2}$ for any $y \in [0,1]$, it follows that $|P_n (a)| \le n^{3+C/2} \delta' = n^{4+C/2} \delta^2$, using the notation in the previous proof. But again the previous proof  provides  that $\P (\exists a\in I, | P_n (a)| \le n^{4+C/2} \delta^2 ) = o(n^{-C})$ .

\end{proof}

\section{Near Double roots and Truncation} 

First of all, we need to truncate the random variables $\xi_0,\dots, \xi_n$. Let $d>0$ be a parameter and let $\CB_d$ be the event $|\xi_0|<n^d \wedge \dots \wedge |\xi_n|<n^d$. As $\xi$ has unit variance, we have the following elementary bound 

$$\P(\CB_d^c)\le n^{1-2d}.$$

In what follows we will condition on $\CB_d$ with $d=2$.

Consider $P_{n}(x) = \sum_{i=0} ^n \xi_i x^i$ and  for $m <n$, we set 

$$g_{m} := P_n -P_m =\sum_{i=m+1} ^n \xi_i x^i.$$  

For any $0< x \le 1-r$, Chernoff's bound yields that  for any $\lambda >0$

$$\P \Big(|g_m(x) | \ge \lambda n^2 \sqrt  { \sum_{i=m +1}^n (1-r)^{2i} }\bigg| \CB_2 \Big ) \le \P \Big (|g_m(x)|  \ge \lambda n^2 \sqrt  { \sum_{i=m}^n x^{2i} }\bigg| \CB_2\Big)  \le 2 \exp( -\lambda^2/2 ). $$

Since  $$\sum_{i=m+1}^n (1-r)^{2i} \le (1-r)^{2m+2 } \frac{1}{1 -(1-r)^2 } := s(r,m),  $$ it follows that 


\begin{equation}\label{eqn:gT}
\P (|g_m| \ge  \lambda n^2 \sqrt{s(r,m)}|\CB_2 ) \le 2 \exp (-\lambda^2/2 ). 
\end{equation}

We next compare the roots of 
$P_n$ and $P_m$ in the interval $(0, 1-r)$. Our intuition is that if $s(r,T)$ is sufficiently small, then there is an bijection $\phi$  between the two sets of roots
such that $x$ and $\phi(x)$ are very close. In particular, the numbers of roots  of two polynomials in this interval are the same with high probability.

\begin{lemma}  \label{approximation} 
 Assume that $F (x) \in C^2(\R)$ and $G(x)$ are continuous functions satisfying the following properties
 
 \begin{itemize}
 \item $F(x_0) =0$ and $|F'(x_0)| \ge \epsilon_1$;
 \vskip .1in
 \item  $|F^{''} (x)  | \le M$ for all $x \in I := [x_0 -\epsilon_1 M^{-1} , x_0 +\epsilon_1 M^{-1} ]$; 
 \vskip .1in
 \item $\sup_{x \in I} | F(x)- G(x) | \le \frac{1}{4} \epsilon_1^{2} M^{-1}$. 
 \end{itemize}
 Then $G$ has a root in $I$.
\end{lemma}

\begin{proof}[Proof of Lemma \ref{approximation}] 
  We can assume, without loss of generality, that 
$G(x_0) \ge 0$.   Consider two cases:

 {\bf Case 1.}  
 $F'(x_0)  \ge \epsilon_1 $. Using the bound $|F^{''} (x)| \le M$ and the mean value theorem, it follows that $F' (x) \ge \frac{1}{2} \epsilon_1$ for all $x$ satisfying $x_{-} := x_0 - \frac{1}{4} \epsilon_1 M^{-1}  \le x \le x_0$. It follows that 
$F(x_{-} ) \le - \frac{1}{4} \epsilon_1^2  M^{-1} $.  Thus, $G(x_{-})  \le 0$ and so $G$ must have a root between $x_{-}$ and $x_0$.

{\bf Case 2.} $F'(x_0) \le - \epsilon_1 $. Arguing similarly, we can prove that $G$ has a root between $x_0$ and $x_{+} := x_0 + \frac{1}{4} \epsilon_1 M^{-1}$. 
\end{proof}

By combining Lemma \ref{approximation} and Lemma \ref{lemma:dbroot}, we obtain the following key observation.

Set $B:=\max(B_1, B(2),8)$, where $B_1,B(2)$ are the constants from Theorem \ref{lemma:dbroot} and Theorem \ref{lemma:distance} corresponding to $C=2$.

\begin{lemma}[Roots comparison for truncated polynomials]\label{lemma:comparison} Let $r \in (B_1 \log n/n, B_0^{-1}]$ and $m=4B r^{-1} \log n$. Then for any  subinterval $J$ of $(1-B_0^{-1},1-r)$ one has

\begin{equation} \label{Tn3} | \E N_n J  - \E N_m J |   \le  m^{-1} . 
\end{equation} 
\end{lemma}

\begin{proof}[Proof of Lemma \ref{lemma:comparison}] 


Condition on $\CB_2$, one has $\sup_{|x|\le 1}\max(|P_n^{''} (x)|, |P_m^{''} (x)|) \le n^5$ with probability one. Set  $\lambda := \log n$,  by \eqref{eqn:gT}, with probability at least $1 - 2\exp( -\log^2 n/2 ) \ge 1 - n^{-\omega (1) }$ the following holds


$$|P_n (x)- P_m(x) | \le \lambda n^2 \sqrt{s(r, m)} = \lambda n^2 (1-r)^{m+ 1} \frac{1}{\sqrt{1 -(1-r)^2} } \le n^{-3B}  $$  

for all $0 \le x \le 1-r$.

By Lemma \ref{lemma:dbroot} (with $C=2$), $|P_n '(x) | \ge n^{-B}$ for all $x\in J$ with probability $1-o( n^{-2})$. Applying Lemma \ref{approximation} with $\epsilon_1 =n^{-B}, M= n^5$, $F= P_n, G = P_m$, we conclude that with probability $1-o(n^{-2})$, for any root $x_0$ of $P_n(x)$ in the interval $(1- B_0^{-1}, 1-r)$ (which is a subset of $(1-B_0^{-1}, 1-B_1 \log n/n)$), there is a root $y_0$ of $P_m (x)$ such that $|x_0 -y_0| \le \epsilon_1 M^{-1}=  n^{-B-5}$.

On the other hand, applying  Lemma \ref{lemma:distance} with $C=2$, again with probability $1 -o(n^{-2})$ there is no pair of roots of $P_n$ in $J$ with distance less than $n^{-B}$. It follows that  for different roots $x_0$ we can  choose different roots $y_0$. Furthermore, by the second part of  Lemma \ref{lemma:distance}, with probability $1-o(n^{-2})$, all roots of 
$P_n(x)$ must be of distance at least $n^{-B} $ from the two ends of the interval.  If this holds, then all $y_0$ must also be inside the interval. 
This implies that with probability at least $1 -o(n^{-2})$,   the number of roots of $P_m$ in $J$  is at least that of $P_n$. Putting together, we obtain 

\begin{equation} \label{Tn1} \E N_m J  \ge \E N_n J -  (o(n^{-2})+n^{-3}) n  \ge \E N_n J - n^{-1} , \end{equation} where the extra term $n^{-1}$ comes from the fact that $P_n$ has at most $n$ real roots.

Switching the roles of $P_n$ and $P_m$, noting that as $r=4B\log n/m \ge B_1\log m/m$, 

$$J \subset (1- B_0^{-1}, 1-r) \subset (1-B_0^{-1}, 1-B_1 \log m/m).$$ 

As such, Lemmas \ref{lemma:dbroot} and \ref{lemma:distance} are also applicable to $P_m(x)$. Argue similarly as above, we also have

\begin{equation} \label{Tn2} \E N_n J  \ge \E N_m J -  (o(m^{-2})+n^{-3}) m  \ge  \E N_m J - m^{-1}. \end{equation} 

It follows that 

$$| \E N_n J  - \E N_m J |   \le   m^{-1}.$$

\end{proof}


\begin{remark}
By applying Lemma \ref{lemma:dbroot} and Lemma \ref{lemma:distance} to higher values of $C$, with sufficiently large $B_0$ and $B_1$ one obtains the following bound for any interval $J$ of $(1-B_0^{-1},1-r)$,
\begin{equation}\label{Tn4}
| \E N_n J  - \E N_m J | \le m^{-C}.
\end{equation}
However, in later application $C=1$ would be sufficient.
\end{remark}

\begin{remark}
If one can show Lemma \ref{lemma:dbroot} for $B_0=1$, then Lemma \ref{lemma:comparison} is true for any $J\subset (0,1-r)$.
\end{remark}

\section{Proof of Theorem  \ref{theorem:main1}} 

We first prove \eqref{Kac11}. Let $C_0 =\max(C_1,B_0^{1/\ep})$ where $C_1$ is the constant in  Theorem \ref{theorem:TV} and let $C$ be any number greater than $C_0$.

Let $\ep>0$ be a small constant to be chosen. Set $n_0:=n, r_0 = n^{-\ep}$ and define recursively 

$$n_i :=  4B  r_{i-1}^{-1}  \log n_{i-1}, \mbox{ and } r_i := n_i^{-\ep}, i\ge 1.$$ 

It is clear that $\{n_i\}$ and $\{r_i\}$ are respectively decreasing and increasing sequences . Let $L$ be the largest index such that $n_L \ge C$. By definition, $r_i\le B_0^{-1}$ for all $1\le i\le L$. Also, as $C > n_{L+1} = 4B n_L^{\ep} \log n_L\ge n_L^{\ep}$, it follows that $n_L <  C ^{1/ \epsilon}$. Thus,

\begin{equation}\label{b2}
n_L \in[C,   C ^{1/ \epsilon}].
\end{equation}

Set $I_i := (1-r_i, 1-r_{i-1}]$ (with the convention that $r_{-1} =0$). Because $I_i \subset (1-B_0^{-1}, 1-r_{i-1}]\subset (1-B_0^{-1}, 1-r_{j-1}]$ for $1\le j\le i$, by \eqref{Tn3}, 

$$|\E N_{n_{j-1} } I_i  -\E N_{n_{j} } I_i  |  \le n_{j-1}^{-1} . $$

By the triangle inequality,

\begin{equation}\label{eqn:triangle:1}
|\E N_{n_0} I_i  -\E N_{n_{i} } I_i  |  \le \sum_{j=1}^i n_{j-1}^{-1} \le 2 n_{i-1}^{-1}. 
\end{equation}

On the other hand, as $n_i \ge C_1$ for $i\le L$, by Theorem \ref{theorem:TV} 

\begin{equation}\label{eqn:triangle:2}
|\E N_{n_{i} } I_i - \E N_{n_{i} , N(0,1)} I_i  | \le n_i^{-\alpha} . 
\end{equation}

Combining \eqref{eqn:triangle:1} and \eqref{eqn:triangle:2}, one obtains

\begin{equation}\label{eqn:triangle:3}
|\E N_{n_0} I_i  -\E N_{n_0,N(0,1) } I_i  |  \le 2n_{i-1}^{-1} + n_i^{-\alpha}. 
\end{equation}

Let $I =\cup_{i=0}^L I_i$, again by the triangle inequality

$$|\E N_n I - \E N_{n, N(0,1)} I |   \le 2 \sum_{i=0}^n n_i^{-1} +  \sum_{i=0}^n n_i^{-\alpha}. $$

The left end point of $I$ is $1 -n_L^{-\epsilon} = 1 -{C'} ^{-1}$, where $C' := n_L ^{\epsilon} \in [C^{\ep}, C]$ by \eqref{b2}. Furthermore, by definition of the $n_i$, it is easy to show that

$$\sum_{i=0}^L n_i^{-\alpha } \le 2 n_L^{-\alpha} =o( {C'} ^{-1} ),$$ 

assuming (without loss of generality) that $\epsilon < \alpha/2$. 

Thus, we can conclude  that there exists $C' \in  [C^{\ep}, C]$ such that for 
$I: =(1- {C'} ^{-1} ,1 ]$,

$$| \E N_n I - \E N_{n, N(0,1)} I | \le {C'} ^{-1} ,$$ concluding the proof of \eqref{Kac11}.


\begin{remark}\label{rmk1}
Notice that Theorem \ref{theorem:TV} holds for any subinterval of the form $(1-r,a)$ where $r \le n^{-\epsilon}$. Thus, one can prove the same bound for $I$ being any subinterval of $(1- {C'}^{-1} ,1 ]$ by setting $I_i := (1-r_i, 1-r_{i-1}] \cap I$ in the above argument. As a consequence, by choosing $C=C_0$ and $ I=(1- {C_0}^{-1} ,1 ]$ one obtains 

$$| \E N_n (1- {C_0}^{-1} ,1 ] - \E N_{n, N(0,1)} (1- {C_0}^{-1} ,1 ] | \le  {C'}^{-1}\le C_0^{-\ep},$$

proving \eqref{Kac10}.
\end{remark}

\section{Proof of Lemma \ref{fact:edge}}

\subsection{Justification of \eqref{Kac11-0}}

%
%
%
%
%
%
%
%
%
%
%
%
%
%

We follow the approach  developed in  \cite{IM2}.
First, there exist some constants $q_1, c\in (0, 1)$ depending only on $\ep$ and $T$, where $T$ is an upper bound of $\E|\xi|^{2+\ep}$, such that $\P(|\xi|\le c)= q \le q_1< 1$. Indeed, put $p = \P(|\xi|> c)$, then

\begin{eqnarray}
1 = \E|\xi|^{2}&=&\E\left(|\xi|^{2}, |\xi|\le c\right)+\E\left(|\xi|^{2}, |\xi|>c\right)\nonumber\\
&\le& c^{2} + \E\left(|\xi|^{2+\epsilon}\right)^{\frac{2}{2+\epsilon}}\P\left(|\xi|>c\right)^{\frac{\epsilon}{2+\epsilon}}\nonumber\\
&\le& c^{2} + p^{\frac{\epsilon}{2+\epsilon}}T^{\frac{2}{2+\epsilon}}.\nonumber
\end{eqnarray}

Thus, by choosing $c$ small, we get $p$ greater than some positive amount.\\

Next, let

$$B_k = \Big \{\omega: \ab{\xi_0}\le c,\dots, \ab{\xi_{k-1}}\le c, \ab{\xi_k}>c \Big \}, \mbox{ where } k = 0, \dots, n+1.$$


Then $\P(B_k) = (1-q)q^{k}$.
Note that if $P_n$ has $N$ zeros in $[-1+\frac{1}{C}, 1-\frac{1}{C}]$ then $P_n^{(k)}$ has at least $N-k$ zeros in that interval. Thus, 

\[N_{P_n}{[-1+\frac{1}{C}, 1-\frac{1}{C}]}\le k +N_{P_n^{(k)}}{ [-1+\frac{1}{C}, 1-\frac{1}{C}]}.\]

By Jensen's inequality for $P_n^{(k)}$, 

\begin{eqnarray}
N_{P_n}{[-1+\frac{1}{C}, 1-\frac{1}{C}]}\le k +\frac{\log \frac{M_k}{P_n^{(k)}(0)}}{\log \frac{R}{r}}, \nonumber
\end{eqnarray}

where $R = 1 - {\frac{1}{2C}}, r = 1-\frac{1}{C}$, and $M_k = \sup _{|z| = R} \ab{P_n^{(k)}(z)}$. 

Conditioned on $B_k$, we have 

$$ P_n^{(k)}(0) = k!\ab{\xi_k}> k!c, \mbox{ and } M_k\le \sum_{j=k}^{n}j(j-1)\dots(j-k+1)|\xi_j|R^{j-k}.$$ 

Thus, on $B_k$, 
\begin{eqnarray}
N_{P_n}{[-1+\frac{1}{C}, 1-\frac{1}{C}]}&\le& k +\frac{\log \frac{\sum_{j=k}^{n}j(j-1)\dots(j-k+1)|\xi_j|R^{j-k}}{k!c}}{\log \frac{R}{r}}\nonumber\\
&=&k +\frac{\log \frac{\sum_{j=k}^{n}c_{jk}|\xi_j|}{c}}{\log \frac{R}{r}}\nonumber,
\end{eqnarray}

where 

\begin{equation}\label{eqn:c_{jk}}
c_{jk} = j(j-1)\dots(j-k+1)R^{j-k}/k!.
\end{equation}

So,

\begin{eqnarray}
\E N_n{[-1+\frac{1}{C}, 1 - \frac{1}{C}]}&\le& \sum_{k=0}^{n+1} k\P(B_k) + \frac{1}{\log\frac{R}{r}} \sum_{k=0}^{n+1} \int_{B_k} \log\left(\sum_{j=k}^{n}c_{jk}|\xi_j|\right)\text{d}\P - \frac{\log c}{\log\frac{R}{r}}\sum_{k=0}^{n+1} \P(B_k).\nonumber
\end{eqnarray}

Since $\sum_{k=0}^{n+1} k\P(B_k) = (1-q)q\sum_{k=0}^{\infty}kq^{k-1}=\frac{q}{1-q}\le\frac{q_1}{1-q_1}$, and $\log \frac{R}{r} =\log\left(1 + \frac{1/2C}{1 - 1/C}\right)\ge \frac{1}{4C}$, the proof is complete if we can show the following claim.

\begin{claim}\label{claim:sumB_k} There exists a constant ${C}''$ such that 
\begin{equation}\label{c1}
\sum_{k=0}^{n+1} \int_{B_k} \log\left(\sum_{j=k}^{n}c_{jk}|\xi_j|\right)\text{d} \P\le  C''.
\end{equation}
\end{claim}


\begin{proof}[Proof of Claim \ref{claim:sumB_k}] Let $X_k = \sum_{j=k}^{n}c_{jk}|\xi_j|$, where we recall $c_{jk}$ from \eqref{eqn:c_{jk}}, and let $Z_k = \E X_{k}$. Then
\begin{eqnarray}\label{c2}
Z_k\le {\E|\xi|}\sum_{j=k}^{\infty}c_{jk} = \frac{\E|\xi|}{(1-R)^{k+1}}\le \frac{1}{(1-R)^{k+1}} = (2C)^{k+1}.
\end{eqnarray} 
Let $B_{ki} = \{\omega\in B_k: e^{i}Z_k\le X_k\le e^{i+1}Z_k\}$. 


Then $\P(B_{ki})\le e^{-i}$ by Markov's inequality. Let $i_0 = \lfloor-\log P(B_k)\rfloor$, then

\begin{eqnarray}
&&\int_{B_k}\log X_k\text{d}\P \le \P(B_k\setminus B_{ki_0})\log \left(e^{i_{0}}Z_k\right)+\sum_{i = i_{0}}^{\infty} \int_{B_{ki}}\log X_k\text{d}\P\nonumber\\
&\le&\P(B_k)\log \left(\frac{Z_k}{\P(B_k)}\right)+\sum_{i = i_{0}}^{\infty} \log \left(e^{i+1}Z_k\right)e^{-i}\nonumber\\
&\le& \P(B_k)\Big((k+1)\log (2C)-\log \P(B_k)\Big) + (k+1)\log (2C)\sum_{i = i_{0}}^{\infty}e^{-i}+\sum_{i = i_{0}}^{\infty} (i+1)e^{-i}\quad\text{by \eqref{c2}}\nonumber\\
&\le&\P(B_k)\Big((k+1) C'-\log \P(B_k)\Big) + (k+1) C' e^{-i_0}+ C'(i_0+1)e^{-i_0} \nonumber\\
&\le& \P(B_k)\Big ((k+1) C'-\log \P(B_k)\Big ) + (k+1) C'\P(B_k)+ C'(1-\log \P(B_k))\P(B_k) \nonumber\\
&\le& C'\P(B_k)\Big(k+1-\log \P(B_k)\Big)\nonumber.
\end{eqnarray}

Thus, 
\begin{eqnarray}
\sum_{k=0}^{n+1} \int_{B_k} \log\left(\sum_{j=k}^{n}c_{jk}|\xi_j|\right)\text{d}\P&\le& C'\sum_{k=0}^{n+1} q^{k}(1-q)(k+1 - k\log q - \log(1-q))\nonumber\\
&\le& C'\sum_{k=0}^{\infty} q_1^{k}(k+1 - \log(1-q_1)) + C'\left(\log\frac{1}{q} \right)\sum_{k=0}^{\infty}kq^{k}\nonumber\\
&\le& C' + C'\left(\log\frac{1}{q} \right)\frac{q}{(1-q)^{2}} \nonumber\\
&\le& C' + C'\left(\log\frac{1}{q_{1}} \right) \frac{q_1}{(1-q_1)^{2}}\nonumber.
\end{eqnarray}
This proves (\ref{c1}) and completes the proof.

\end{proof}


%
%
%
%
%
%
%
%
%
%
%
%
%
\subsection{Justification of \eqref{Kac11-1}}
Let $C_0$ as in the proof of Theorem \ref{theorem:main1}. By Remark \ref{rmk1}, 

$$| \E N_n I - \E N_{n, N(0,1)} I | \le  {C_0}^{-\ep}\le 1 ,$$

where $I$ is any subinterval of  $[1-\frac{1}{C_0}, 1]$. 

Let $C$ be any number greater than $C_0$, and let $I = [1-\frac{1}{C_0}, 1 - \frac{1}{C})$, then 

$$| \E N_nI - \E N_{n, N(0,1)} I | \le 1.$$

Combining this with the bound in \eqref{Kac11-0} for $C_0$, we obtain

$$\E N_n [0, 1-\frac{1}{C})\le \E N_{n, N(0,1)}I + M(C_0) + 1\le \E N_{n, N(0,1)}[0, 1-\frac{1}{C}) + M(C_0) + 1.$$

Now, by the Edelman-Kostlan formula (see \cite{EK}), 

\begin{eqnarray}   
\E N_{n, N(0,1)}  [0, 1-\frac{1}{C})  &=& \frac{1}{\pi}\int_{0}^{1-\frac{1}{C}}   \sqrt { \frac{1}{(1 - x^2)^2} - \frac{(n+1)^2 x^{2n} }{ (1 - x^{2n+2})^2 } } \text{d} x\nonumber\\
&\le& \frac{1}{\pi}\int_{0}^{1-\frac{1}{C}}   \frac{1}{1 - x^2} \text{d} x = \frac{1}{2\pi}\left(\log C + \log \left(2-\frac{1}{C}\right)\right)\nonumber\\
&\le& \frac{1}{2\pi}\log C + 1\nonumber.
 \end{eqnarray} 
 
Thus, 

$$\E N_n [0, 1-\frac{1}{C})\le\frac{1}{2\pi}\log C + M(C_0) + 2.$$

This proves \eqref{Kac11-1}.

\end{document}